\documentclass[12pt]{article}
\usepackage[english]{babel}
\usepackage[latin1]{inputenc}
\usepackage[T1]{fontenc}
\usepackage{amsmath,amssymb,amsthm,xspace}

\RequirePackage{ifpdf}
\ifpdf
   \usepackage[pdftex]{hyperref}
\else
   \usepackage[hypertex]{hyperref}
\fi

\title{On the isometry group of the Urysohn space}
\author{Katrin Tent and Martin Ziegler}
\date{February 3, 2012}

\newtheorem{theorem}{Theorem}[section]
\newtheorem{example}[theorem]{Examples}
\newtheorem{lemma}[theorem]{Lemma}
\newtheorem{proposition}[theorem]{Proposition}
\newtheorem{corollary}[theorem]{Corollary}
\newtheorem{definition}[theorem]{Definition}
\newtheorem{remark}[theorem]{Remark}

\newcommand{\nc}{\newcommand}
\nc{\inv}{^{-1}}
\nc{\Q}{\mathbb{Q}}
\nc{\R}{\mathbb{R}}
\nc{\C}{\mathcal{C}}
\nc{\G}{\mathcal{G}}
\nc{\M}{\mathcal{M}}
\nc{\U}{\mathbb{U}}
\nc{\Frl}{Fra\"iss\'e limit\xspace}
\nc{\Frls}{Fra\"iss\'e limits\xspace}
\nc{\fg}{finitely generated\xspace}

\renewcommand{\phi}{\varphi}

\DeclareMathOperator{\Fix}{Fix}
\DeclareMathOperator{\aut}{Aut}
\DeclareMathOperator{\diam}{diam}
\DeclareMathOperator{\tp}{tp}
\DeclareMathOperator{\dom}{dom}
\DeclareMathOperator{\im}{im}

\newcommand{\Ind}{
 \setbox0=\hbox{$x$}\kern\wd0\hbox to 0pt{\hss$
 \mid$\hss}\lower.9\ht0\hbox to 0pt{\hss$\smile$\hss}\kern\wd0
}
\newcommand{\indep}[3]{#1\mathop{\mathpalette\Ind{}}_{#2}#3
}
\newcommand{\Indep}{\indep{}{}{}}
\begin{document}
\maketitle
\begin{abstract}We give a  general criterion for the (bounded)
  simplicity of the automorphism groups of certain countable structures
  and apply it to show that the isometry group of the Urysohn space
  modulo the normal subgroup of bounded isometries is a simple group.
\end{abstract}

\section{Introduction}

Many very homogeneous mathematical structures are known to have
simple, or at least \emph{essentially simple} automorphism
groups. This is true for the complex numbers \cite{Lascar1992}, for
irreducible Riemannian symmetric spaces of noncompact type, and also
as shown in~\cite{MT} for structures arising as a \Frl of a free
amalgamation class.  The subject of the present paper is another very
homogeneous structure, Urysohn's metric space $\U$, which is the unique
complete homogeneous separable metric space which embeds every finite
metric space. It is easy to construct $\U$: it arises as the
completion of the rational metric space obtained as the \Frl $\Q\U$ of
the class of finite metric spaces with rational distances. $\Q\U$ is
called the \emph{rational} Urysohn space and the usual Urysohn space is
sometimes called the \emph{complete} Urysohn space.

Let $G$ denote the isometry group of the (complete) Urysohn space $\U$
and $B$ the normal subgroup of all isometries having bounded
displacement.  We will show that the quotient $G/B$ is a simple group.

This will follow from a more general result on automorphism groups of
countable structures with a certain independence relation. As another
application of this general result we will give another proof that for classes
with free amalgamation in relational languages the automorphism
group is a simple group unless the \Frl is an indiscernible set.

Our formal framework will be introduced in Section~\ref{S:terminology}
with the main theorem proved in Section~\ref{Sec:proof}. A detailed
analysis of unbounded isometries in the Urysohn space in Section~\ref{Sec:ury} then allows us to apply our main result to the
Urysohn space. As another application we recover and sharpen the results of \cite{MT} in Section~\ref{Sec:free}.

\section{Terminology and notation}\label{S:terminology}

Let $\M$ be a structure and $\G$ its automorphism group.  Using model
theoretic language, for a tuple $\bar a$ and a finite set $B$ we say
that the tuple $\bar a'$ \emph{realises the type} $p=\tp(\bar a/B)$ if there is
an automorphism of $\M$ which maps $\bar a$ to $\bar a'$ and fixes $B$
pointwise.\footnote{If $\M$ is countable and $\omega$-saturated, the
  types so defined correspond exactly to types in the model theoretic
  sense. If $\M$ is only $\omega$--homogeneous, they correspond to
  \emph{realised} types. And if $\M$ is a \Frl (see below), they
  correspond to realised quantifier free types.}

Let $\indep{A}{B}{C}$ be a ternary relation between finite subsets of
$\M$, pronounced $A$ and $C$ are independent over $B$.

\begin{definition}\label{D:ind}
  $\Indep$ is a stationary independence relation if the following
  axioms are satisfied.
  \begin{enumerate}
  \item {\rm (Invariance)} $A$ and $B$ being independent over $C$
    depends only on the type of $ABC$. (We choose enumerations for
    $A$, $B$ and $C$ and consider them as tuples. Note that we write
    $AB$ for the union $A\cup B$.)
  \item {\rm(Monotonicity)} \[\indep{A}{B}{CD} \;\text{ implies
  }\;\indep{A}{B}{C}\text{ and }\indep{A}{BC}{D}.\]
  \item
    {\rm(Transitivity)}
    \[\indep{A}{B}{C}\text{ and }\indep{A}{BC}{D}\;
    \text{ implies }\;\indep{A}{B}{D}.\]
  \item        {\rm(Symmetry)}
    \[\indep{A}{B}{C}\;\text{ implies }\;\indep{C}{B}{A},\]

  \item {\rm (Existence)}\label{R:existence} Let $p$ be a type over $B$
    and $C$ a finite set. Then $p$ has a realisation which is
    independent from $C$ over $B$.
  \item {\rm(Stationarity)} If $\bar a$ and $\bar a'$ have the
    same type over $B$ and are both independent from $C$ over $B$,
    then $\bar a$ and $\bar a'$ have the same type over $BC$.
  \end{enumerate}
If $\indep{A}{B}{C}$ is only defined for non-empty $B$, we say that
$\Indep$ is a \em{local} independence relation on $\M$.
\end{definition}
\noindent It is easy to see that the axioms imply
\[\indep{A}{B}{C}\;\Leftrightarrow\;\indep{AB}{B}{C}
\;\Leftrightarrow\;\indep{A}{B}{BC}.\] Also, on the basis of the other
axioms Stationarity follows from the following special case for single
elements:\\

\quad\parbox{30em}{(Stationarity') \it If $a$ and $a'$ have the same type
  over $B$ and are both independent from $c$ over $B$, then $a$ and
  $a'$ have the same type over $Bc$.}\\

\begin{example}\label{examples}\upshape
  \indent\par
  \begin{enumerate}
  \item
    By a well-known construction of Fra\"iss\'e, a countable class
    $\C$ of \fg structures, closed under \fg substructures and satisfying the
    amalgamation and joint embedding properties has a \Frl: this is a
    countable structure $\M$ whose \fg substructures are - up to
    isomorphism - exactly the elements of $\C$ and which has the
    property that any isomorphism between \fg substructures extends to
    a global automorphism of $\M$ (see~\cite{TZ}, Ch.\ 4.4 for more
    details).

    In many cases the amalgamation property of $\C$ is verified by the
    existence of a ``canonical'' amalgam $X\otimes_Y Z$ of $X$ and $Z$
    over the common substructure $Y$ which is functorial in the sense
    that automorphisms of the factors $X$ and $Z$ fixing $Y$
    elementwise will extend to the amalgam. This can then be
    used to define two finite subsets $A$ and $C$ of $\M$ to be
    independent over $B$ if $\langle A\cup B\cup C\rangle$ is
    isomorphic to $\langle A\cup B\rangle \otimes_{\langle
      B\rangle}\langle B\cup C\rangle$ under an isomorphism
    commuting with the embeddings, where $\langle S\rangle$
    denotes the substructure generated by $S$. At this level of
    generality, the independence notion
    satisfies only Existence, Invariance and Stationarity.
    In the following cases it defines either a stationary independence
    relation, or a local stationary independence relation, and in the local
    case it suffices to have $\indep{A}{B}{C}$ defined for $B$ non-empty:
    \begin{enumerate}
    \item
      The class $\C$ of finite metric spaces with distances in a countable
      additive subsemigroup $R$ of the positive reals has canonical
      amalgamation over a nonempty base: If
      $B$ is non-empty and $A$ and $C$ are two extensions of $B$ which
      intersect exactly in $B$, we can put $A\otimes_BC=A\cup C$ with
      the metric defined by
      \[d(a,c)=\min\{d(a,b)+d(b,c)\colon b\in B\}\] if $a\in B,c\in
      C$. The \Frl is the $R$-valued Urysohn space $R\U$. Then
      $\indep{A}{B}{C}$ if and only if for all $a\in A,c\in C$ there
      is some $b\in B$ with $d(a,c)=d(a,b)+d(b,c)$. Note that
      independence over the empty set is not defined.  The complete
      Urysohn space $\U$ is the completion of $\Q\U$.

    \item The bounded Urysohn space $\U_1$ enjoys similar
      properties with respect to the class of  finite metric spaces
      with diameter at most~$1$ and is constructed in a
      similar fashion, as the completion of a Fraisse limit. We
      let $A\otimes_BC$ denote the metric space such that for $a\in
      A,c\in C$ the distance of $a$ and $c$ is the minimum
      of \[\{d(b,a)+d(b,c)\colon b\in B\}\cup\{1\}.\] Here $B$ may be
      empty.

    \item\label{free} If $\C$ is a class of relational structures, we
      may put $A\otimes_B C$ as the \emph{free amalgam}, i.e.\ the
      structure on the set $A\cup C$ with no new relations on
      $(A\setminus B)\cup(C\setminus B)$. Then $A$ and $C$ are
      independent over $B$ if and only if whenever $R(d_1,\ldots,d_n)$
      holds for elements $d_1,\ldots d_n$ of $B\cup A\cup C$ then
      either all $d_i$ are in $B\cup A$ or all $d_i$ are in $B\cup
      C$. The random graph and random hypergraphs, the $K_n$-free
      graphs and their hypergraph analogs arise in this way. Again $B$ may be empty here.
    \end{enumerate}
  \item Let $T$ be a stable complete theory and $\M$ an
    $\omega$-homogeneous countable model on $T$.  Then
    forking--independence has all properties of Definition~\ref{D:ind}
    except possibly Stationarity, see \cite[Ch.\ 8.5]{TZ}. For
    Stationarity we have to assume that all $1$--types are stationary,
    which implies that all types are stationary. (An example of such a theory
    is the theory of trees with infinite valency.)
  \end{enumerate}
\end{example}

To see that Transitivity holds in the Urysohn spaces assume
$\indep{A}{B}{C}$ and $\indep{A}{BC}{D}$ and consider $a\in A$ and
$d\in D$. By assumption there is some $x\in BC$ with
$d(a,d)=d(a,x)+d(x,d)$. If $x\in C$, there is some $b\in B$ with
$d(a,x)=d(a,b)+d(bx)$. This implies $d(a,d)=d(a,b)+d(b,d)$, as
required. The rest is clear.

The independence relations in examples 1(a), 1(b) and in 1(c) for
binary relations have stronger properties than forking-independence
has in general, notably
\[\indep{A}{B}{C} \text{ and } B\subset B'\; \text{ implies }\;
\indep{A}{B'}{C}.\]
However, our proofs do not make use of these additional properties.

\begin{definition}
  We say that a finite tuple $\bar x$ is \em{independent} from a tuple
  $\bar y$ over $A;B$ if
  \[\indep{\bar x}{A}{B\bar y}\;\text{ and }\;
  \indep{\bar xA}{B}{\bar y}.\]
\end{definition}
\begin{lemma}\label{L:ind}
  Let $\Indep$ be a stationary independence relation on $\M$.
  Then the following holds.
  \begin{enumerate}
  \item\label{ind_krit} For $\bar x$ to be independent from $\bar y$
    over $A;B$ it is enough to have $\indep{\bar x}{A}{B}$ and
    $\indep{\bar xA}{B}{\bar y}$.
  \item\label{ind_ex} \textup{(Existence)} Let $p$ be a type over $A$
    and $q$ a type over $B$. Then there are realisations $\bar x$ of
    $p$ and $\bar y$ of $q$ such that $\bar x$ is independent from
    $\bar y$ over $A; B$. The type $\tp(\bar x\bar y/AB)$ is uniquely
      determined.
  \item\label{ind_trans} \textup{(Transitivity)} If $\bar x$ is
    independent from $\bar y$ over $A; B$ and $\bar x'$ is independent
    from $\bar y'$ over $\bar xA,\bar yB$, then $\bar x\bar x'$ is
      independent from $\bar y\bar y'$ over $A; B$.
  \item\label{ind_sym} \textup{(Symmetry)} If $\bar x$ is independent
    from $\bar y$ over $A,B$, then $\bar y$ is independent from $\bar
    x$ over $B; A$.
  \end{enumerate}
\end{lemma}
\begin{proof}\indent\par\noindent
  \ref{ind_krit}. By Symmetry and Monotonicity $\indep{\bar xA}{B}{\bar
    y}$ implies $\indep{\bar x}{AB}{\bar y}$. By Transitivity and
  $\indep{\bar x}{A}{B}$ this implies $\indep{\bar x}{A}{B\bar y}$.\\

  \noindent\ref{ind_ex}. Choose $\bar x$ such that $\indep{\bar
    x}{A}{B}$ and then $\bar y$ such that $\indep{\bar xA}{B}{\bar
    y}$.\\

  \noindent\ref{ind_trans}. Note that $\indep{\bar xA}{B\bar y}{\bar
    y'}$ implies $\indep{\bar x}{AB\bar y}{\bar y'}$.  From
  $\indep{\bar x}{A}{B\bar y}$ and Transitivity we get $\indep{\bar
    x}{A}{B\bar y\bar y'}$. This and $\indep{\bar x'}{\bar xA}{B\bar
    y\bar y'}$\/ imply $\indep{\bar x\bar x'}{A}{B\bar y\bar y'}$ by
  Transitivity. Similarly one proves $\indep{\bar x\bar x'A}{B}{\bar
    y\bar y'}$.\\

  \noindent\ref{ind_sym}. This follows directly from the symmetry of
  $\Indep$.
\end{proof}
\begin{definition}
  Let $\Indep$ be a (local) independence relation on $\M$ and
  $g\in\G$.  For a finite set $X$ and $p$ a type over $X$ we say that
  $g$ moves a realisation $\bar x$ of $p$ maximally if $\bar x$ is
  independent from $g(\bar x)$ over $X; g(X)$. We say that $g$ moves
  maximally if for all (non-empty) finite sets $X$ and all types $p$
  over $X$, $g$ moves some realisation of $p$ maximally.
\end{definition}

\noindent Note that part~(\ref{ind_sym}) of Lemma \ref{L:ind} implies
that $g$ moves maximally if and only if $g\inv$ does.

If $\M$ is the countable infinite set with no structure, $A$ and $C$
are independent over $B$ if $A\cap C\subset B$. Hence a permutation of $\M$
moves maximally if and only it has infinite support. More generally if
$\M$ is an $\omega$--saturated countable strongly minimal structure
in which algebraic and definable closure coincide
and  $\indep{}{}{}$ is algebraic (i.e.\ forking-) independence, then
$g$ moves maximally if and only if $g$ is unbounded in the sense of
\cite{Lascar1992}. To see this note that an automorphism $g$ of a strongly minimal
structure $\M$ is bounded in the sense of \cite{Lascar1992} if and only if there is a finite set $X$ such that
for any $a\in\M$ we have $g(a)\in acl(aX)$.

\begin{lemma}
  For an automorphism to move maximally it suffices to move
  realisations of $1$-types (i.e.\ types of single elements)
  maximally.
\end{lemma}
\begin{proof}
  This follows from Lemma \ref{L:ind}(\ref{ind_trans}).
\end{proof}

Here is our main result, which will be proved in Section
\ref{Sec:proof}:

\begin{theorem}\label{T:main}
  Suppose that $\M$ is a countable structure with a local stationary
  independence relation and let $g\in\G=\aut(\M)$ move maximally. If
  $\G$ contains a dense conjugacy class, then any element of $\G$ is
  the product of eight conjugates of $g$.
\end{theorem}

We note that for a structure with a stationary independence relation
the assumption that $\G$ contains a dense conjugacy class is always
satisfied:

\begin{lemma}\label{L:kechris_rosendal}
  The automorphism group of a countable structure $\M$ with a stationary
  independence relation has a dense conjugacy class. The same holds for a local
  stationary independence relation if for all finite tuples $\bar a$
  and $\bar b$ in $\M$ realising
  the same type there is some element $c$ with
  $\tp(\bar a/c)=\tp(\bar b/c)$.
\end{lemma}
\begin{proof}
  It is immediate that $\G$ contains a dense conjugacy class if and
  only the following is true: given finite tuples $\bar x,\bar y,\bar
  a, \bar b$ with $\tp(\bar x)=\tp(\bar y)$ and $\tp(\bar a)=\tp(\bar
  b)$ there are tuples $\bar x', \bar y'$ such that $\tp(\bar x'\bar
  y')=\tp(\bar x\bar y)$ and $\tp(\bar x'\bar a)=\tp(\bar y'\bar
  b)$. If $\M$ has a stationary independence relation, then we can
  choose $\bar x'y'$ realising $\tp(\bar x\bar y)$ with $\indep{\bar
    x'y'}{}{\bar a\bar b}$.  By stationarity we then have $\tp(\bar
  x'\bar a)=\tp(\bar y'\bar b)$.

  If $\M$ has a local independence relation, let $c$ be such that
  $\tp(\bar a/c)=\tp(\bar b/c)$ and work over $c$.
\end{proof}
\begin{corollary}\label{C:Urydense}
  The automorphism group of the Urysohn space has a dense conjugacy class.
\end{corollary}
\begin{proof}
  Just note that given $\bar a,\bar b$ satisfying the same type, we can find a point $c$ at sufficiently large distance from
  $\bar a,\bar b$ such that $\tp(\bar a/c)=\tp(\bar b/c)$.
\end{proof}

\begin{corollary}\label{C:main}
  Suppose that $\M$ is a countable structure with a stationary
  independence relation and let $g\in\G$ move maximally. Then any
  element of $\G$ is the product of eight conjugates of $g$.\qed
\end{corollary}

The following example shows that in Theorem~\ref{T:main} the
assumption that $\G$ contains a dense conjugacy class cannot be
dispensed with:

\begin{example}[Cherlin]\upshape
Let $\C$ be the class of finite bipartite graphs in the language
containing a binary relation presenting the edges and an equivalence
relation with two classes presenting the bipartition. Then $\C$ has
local stationary amalgamation, but for the \Frl $\M$ the automorphism
group $\G$ contains no dense conjugacy class: the normal subgroup $N$
of $\G$ consisting of the automorphisms preserving the equivalence
classes is open. It is the automorphism group of an expansion of $\M$
by a predicate denoting one of the conjugacy classes. In this language
$\C$ has stationary amalgamation. If $g\in \G$ moves maximally and
preserves the equivalence classes, it is an automorphism of this
expanded structure. By Corollary~\ref{C:main}, every element of $N$ is
the product of eight conjugates of $g$.  On the other hand, if
$g\in\G$ does not preserve the equivalence classes, then any
nontrivial commutator $[g,h]$ lies in $N$, showing that $\langle
g\rangle^\G=\G$.
\end{example}

 In Section~\ref{Sec:ury} we
 will show that any isometry of
the Urysohn space with unbounded displacement moves maximally
(Proposition \ref{P:max}) and apply Theorem \ref{T:main} to prove

\begin{theorem}\label{T:urysohn}
  For any unbounded isometry $g$ of the Urysohn space the normal
  subgroup $\langle g\rangle^G$ is all of $G$. In fact, any element of
  $G$ is the product of eight conjugates of $g$. Hence $G/B$
  is a simple group.
\end{theorem}

We have not been able to establish the simplicity of the isometry group
of the bounded Urysohn space with our methods.

\bigskip
The proof  of Theorem \ref{T:main} follows the general strategy of \cite{Lascar1992} and \cite{MT}, using ideas of descriptive set theory. The main technical result
is the following proposition, whose
proof will be given in Section~\ref{Sec:proof}:
\begin{proposition}\label{P:dense}
Under the assumptions and notations of Theorem \ref{T:main},
let \[\phi:\G^4\to\G, \phi:(h_1,\ldots h_4)\mapsto g^{h_1}\ldots g^{h_4}.\]
  Then for any
  open set $U\subseteq\G^4$ there is some open set $W\subseteq\G$ with
  $\phi(U)$ dense in $W$.  Equivalently, for any nowhere dense set $X$
  in $\G$, its preimage $\phi\inv(X)$ is nowhere dense in $\G^4$
\end{proposition}

\noindent
\emph{Proof of Theorem \ref{T:main} from
  Proposition~\ref{P:dense}:}
\noindent By Proposition~\ref{P:dense} the image of $\phi$ is not
meagre, for if $\phi(\G^4)=\bigcup X_i$ with $X_i$ nowhere dense, we
would have $\G^4=\bigcup\phi\inv(X_i)$ contradicting the Baire
Category Theorem.  Note that as the image under an analytic map, the set
$\phi(\G^4)$ has the Baire property and is invariant under conjugation.
 Since by assumption there is a dense conjugacy
class, we conclude
from \cite[Theorem 8.46]{Kechris} (applied to $\G$ acting on
itself by conjugation) that $\phi(\G^4)$ is comeagre. Since $g\inv$
moves maximally as well, the image of $\phi':(h_1,\ldots
h_4)\mapsto(g^{-h_1}\ldots g^{-h_4})$ is also comeagre. So for any
$f\in\G$ the translate $\phi'(\G^4)f$ intersects $\phi(\G^4)$, which
is the claim of Theorem~\ref{T:main}.

\section{Proof of Proposition \ref{P:dense}}\label{Sec:proof}

We continue to work with the countable structure $\M$ with a local
stationary independence relation. $\G$ is the automorphism group of
$\M$.  We write $\Fix(X)$ for the pointwise stabiliser of the set $X$
in $\G$.\\

We start with a simple lemma.

\begin{lemma}\label{L:tool}
  \indent\par
  \begin{enumerate}
  \item\label{tool_triv}\sloppy If $\indep{A}{B}{C}$ and $D$ is arbitrary,
    then there is some $D'$ such that $\tp(D'/BC)=\tp(D/BC)$ and
    $\indep{A}{B}{CD'}$
  \item\label{tool_neu} Let $\indep{A}{B}{C}$ and
    $g_1,\ldots,g_n\in\G$. Then there is $e\in \Fix(BC)$ with
    $\indep{A}{B}{Cg_1^e(C)\ldots g_n^e(C)}$.
  \end{enumerate}
\end{lemma}
\begin{proof}\indent\par\noindent
  \ref{tool_triv}. Choose $D'$ of the right type with
  $\indep{A}{BC}{D'}$ and use Transitivity.\\

  \noindent\ref{tool_neu}. By (\ref{tool_triv}) there are
  $C_1,\ldots,C_n$ with
  \[\tp(C_1,\ldots,C_n/BC)=\tp(g_1(C),\ldots,g_n(C)/BC)\] and
  \[\indep{A}{B}{CC_1\ldots C_n}.\] Choose $e\in \Fix(BC)$ with
  $e(C_i)=g_i(C)$. Then we have $g_i^e(C)=C_i$.
\end{proof}

\begin{proposition}\label{P:zurichtung}
  Consider $g_1,\ldots,g_4\in\G$ and finite sets $X_0,\ldots,X_4$ such
  that $g_i(X_{i-1})=X_i$. Then for $i=1,\ldots 4$ there are $a_i\in \Fix(X_{i-1}X_i)$  and
  extensions $X_i\subset Y_i$ such that
  \begin{enumerate}
  \item $g_i^{a_i}(Y_{i-1})=Y_i$,
  \item $\indep{Y_0}{Y_1}{Y_2}$ and $\indep{Y_2}{Y_3}{Y_4}$.
  \end{enumerate}
\end{proposition}
\begin{proof}
  Step 1. Choose a finite extension $X'_1$ of $X_1$ such that
  $\indep{X_0}{X'_1}{X_2X_3X_4}$, for example $X'_1=X_0\cup\ldots\cup
    X_4$.\\

  \noindent Step 2. Apply \ref{L:tool}(\ref{tool_neu}) to $A=X_0$,
  $B=X'_1$, $C=X'_1X_2X_3X_4$ and the automorphisms $g_2$, $g_3g_2$
  and $g_4g_3g_2$. We obtain $e\in \Fix(X'_1X_2X_3X_4)$ such that
  taking \[X'_2=g_2^e(X'_1),X'_3=g_3^e(X'_2),X'_4=g_4^e(X'_3),\] we
  have \[\indep{X_0}{X'_1}{X'_2X_3X'_4}.\]

  \noindent Step 3. The same argument as in Step 2 yields $f\in
  \Fix(X_0X'_1)$ such that taking \[X'_0=(g_1^f)\inv(X'_1)\] we have
  \[\indep{X'_0}{X'_1}{X'_2X'_3X'_4}.\] Set $h_1=g_1^f$ and $h_i=g_i^e$
  for $i=2,3,4$.\\

  \noindent Step 4. If we apply what we proved so far to the reversed
  sequence $X'_4,X'_3,X'_2,X'_1$, we obtain $b_2,b_3,b_4$ with $b_i\in
  \Fix(X'_{i-1}X'_i)$ and extensions $X'_i\subset Y_i$ for
  $i=1,\ldots,4$ such that \[h_i^{b_i}(Y_{i-1})=Y_i {\mbox {\rm
      \ and} \ \ } \indep{Y_1Y_2}{Y_3}{Y_4}.\]

  \noindent Step 5. Lemma \ref{L:tool}(\ref{tool_triv}) shows that we
  may assume that \[\indep{X'_0}{X'_1}{Y_1Y_2Y_3Y_4}.\] By Monotonicity we
  conclude \[\indep{X'_0}{Y_1}{Y_2Y_3Y_4}.\]

  \noindent Step 6. As in Step 3 we find some $b_1\in
  \Fix(X'_0Y_1)$ such that with \[Y_0=(h_1^{b_1})\inv(Y_1)\] we
  have \[\indep{Y_0}{Y_1}{Y_2Y_3Y_4}.\]
\end{proof}
\begin{remark}\upshape
  In fact, the proof yields slightly more: we
  have \[\indep{Y_1Y_2}{Y_3}{Y_4} \mbox{ \ and
    \ \ }\indep{Y_0}{Y_1}{Y_2Y_3Y_4}\] which together
  imply \[\indep{Y_0Y_1Y_2}{Y_3}{Y_4}.\] Note also that we may choose
  $a_2=a_3$.
\end{remark}
\begin{proposition}\label{P:vier}
  Let $g_1,\ldots,g_4\in\G$ move maximally and let $Y_0,\ldots,Y_4$ be
  finite sets such that $g_i(Y_{i-1})=Y_i$ for $i=1,\ldots 4$. Assume
  also that $\indep{Y_0}{Y_1}{Y_2}$ and $\indep{Y_2}{Y_3}{Y_4}$. Let
  $x_0$ and $x_4$ be two tuples such that $g_4g_3g_2g_1$ maps
  $\tp(x_0/Y_0)$ to $\tp(x_4/Y_4)$. Then for $i=1,\ldots 4$, there are
  $a_i\in \Fix(Y_{i-1}Y_i)$ such that
  \[g_4^{a_4}\ldots g_1^{a_1}(x_0)=x_4.\]
\end{proposition}

For the proof we need two lemmas:

\begin{lemma}\label{L:zug}
  Let $g\in\G$ move maximally, let $X,Y,C$ be finite sets such that
  $g(X)=Y$ and $\indep{X}{Y}{C}$ and let $x$ be a tuple. Then there is some
  $a\in\Fix(XY)$  such that
  \[\indep{g^a(x)}{Y}{C}.\]
\end{lemma}
\begin{proof}
  Let $x'$ be a realisation of $\tp(x/XY)$ moved maximally by $g$ and
  let $a_1\in \Fix(XY)$ be such that $a_1(x')=x$. Then $g^{a_1}$ moves
  $x$ maximally over $XY$. So we have
  \[\indep{x}{XY}{g^{a_1}(x)}.\]

  Now let $y$ be a realisation of $\tp(g^{a_1}(x)/XYx)$
  with \[\indep{y}{xXY}{C}.\] We have then also $\indep{x}{XY}{y}$. By
  Transitivity, Symmetry and the assumption $\indep{X}{Y}{C}$ we
  conclude
  \[\indep{y}{Y}{C}.\]
  Finally choose $a_2\in \Fix(xXY)$ with $a_2g^{a_1}(x)=
  g^{a_1a_2\inv}(x)=y$.
\end{proof}
\begin{lemma}\label{L:neu}
  Let $g\in\G$ move maximally and let $X,Y$ be finite sets with
  $g(X)=Y$. Assume that $x$ and $y$ are tuples with $x$ independent
  from $y$ over $X; Y$ and such that $g(\tp(x/X))=\tp(y/Y)$.  Then
  there is some $a\in \Fix(XY)$ such that
 \[g^a(x)=y.\]
\end{lemma}
\begin{proof}
  Let $x'$ be a realisation of $\tp(x/X)$ which is moved maximally by
  $g$.  Since $\indep{x'}{X}{Y}$, we have
  $\tp(x'/XY)=\tp(x/XY)$. Choose $a_1\in \Fix(XY)$ with
  $a_1(x)=x'$. Then $g^a$ moves $x$ maximally over $X$. Set
  $y'=g^a(x)$.  By Lemma \ref{L:ind}.\ref{ind_ex} we have
  $\tp(xy'/XY)=\tp(xy/XY)$. Choose $a_2\in \Fix(XY)$ with
  $a_2(xy)=a_2(xy')$. Then $g^{a_1a_2}(x)=y$.
\end{proof}
\begin{proof}[Proof of Proposition \ref{P:vier}.]
  Note first that $g_3\inv$ and $g_4\inv$ also move maximally. Two
  applications of Lemma \ref{L:zug} yield
  $a_0\in \Fix(Y_0Y_1)$ and $a_4\in \Fix(Y_3Y_4)$ such that for
  \[x_1=g_1^{a_1}(x_0)\mbox{ \ and\ \ }x_3=(g_4\inv)^{a_4}(x_4)\]
  we have \[\indep{x_1}{Y_1}{Y_2}\mbox{
    \ and\ \ }\indep{Y_2}{Y_3}{x_3}.\] Choose $x_2$ realising the type
  $g_2(\tp(x_1/Y_1))=g_3\inv(\tp(x_3/Y_3))$ (over $Y_2$) and such that
  \[\indep{x_2}{Y_2}{x_1Y_1x_3Y_3}.\] Lemma \ref{L:neu} yields $a_2\in
  \Fix(Y_1Y_2)$ and $a_3\in \Fix(Y_2Y_3)$ such that
  $g_2^{a_2}(x_1)=x_2=(g_3\inv)^{a_3}(x_3)$.
\end{proof}

\begin{proof}[Proof of Proposition~\ref{P:dense}:]
  We suppose $g$ in $\G=\aut(M)$ moves maximally and that $U$
  contained in $\G^4$ is open. We may assume that $U=U_1\times\ldots
  U_4$, where each $U_i$ is a basic open set $U_i=\mathcal U(u_i)$,
  with $u_i$ a finite partial isomorphism and
  \[\mathcal U(u)=\{g\in\G\mid u\subset g\}.\]
  Extend each $u_i$ to some $a_i\in\G$. Then choose finite sets
  $X_0,\ldots,X_4$ such that $\im(u_i)\subset X_i$ and
  $g^{a_i}(X_{i-1})=X_i$ for $i=1,\ldots,4$. We apply Proposition
  \ref{P:zurichtung} to this situation and obtain $b_i\in \Fix(X_{i-1}X_i)$ and
  extensions $X_i\subset Y_i$  with
  $g^{a_ib_i}(Y_{i-1})=Y_i$ and such that
  \[\indep{Y_0}{Y_1}{Y_2}\mbox{\  and \ }
  \indep{Y_2}{Y_3}{Y_4}.\]

  Let $w$ be the finite isomorphism $g^{a_4b_4}\ldots
  g^{a_1b_1}\restriction Y_0$. We set $W=\mathcal U(w)$.

  In order to show that $\phi(U)$ is dense in $W$ we consider a basic
  open subset $\mathcal U(w')$ given by an extension $w\subset w'$.
  Let $x$ be an enumeration of $\dom(w')\setminus Y_0$ and $y=w'(x)$.
  Proposition \ref{P:vier} gives us $c_i\in \Fix(Y_{i-1}Y_i)$ such
  that $g^{a_4b_4c_4}\ldots g^{a_1b_1c_1}(x)=y$. Since $b_i$ and $c_i$
  both fix $\im(u_i)$ pointwise, we have $a_ib_ic_i\in U_i$. So the
  $4$-tuple $(a_1b_1c_1,\ldots,a_4b_4c_4)$ belongs to $U$ and is
  mapped by $\phi$ to $g^{a_4b_4c_4}\ldots g^{a_1b_1c_1}$, which
  belongs to $W'$.
\end{proof}

\section{Application to the Urysohn space}\label{Sec:ury}

We will now apply Theorem~\ref{T:main} to the complete Urysohn space.
We extend our notion of independence to $\U$ in the obvious way: we
write \[\indep{A}{B}{C}\] if and only if for all $a\in A,c\in C$ there
is some $b\in B$ with $d(a,c)=d(a,b)+d(b,c)$.

We first establish the following proposition which may be of interest
in its own right.

\begin{proposition}\label{P:max}
  Any unbounded isometry of the Urysohn space moves maximally.
\end{proposition}

It is easy to see that unbounded isometries exist, i.e.\ that $B$ is a
proper subgroup of $G$ (see also~\cite{CV}, Prop.17). Just define an
automorphism of $g$ on $\Q\U$ by a back-and-forth construction. In the
even steps ensure that $g$ will be everywhere defined and
surjective. In the odd steps ensure that there are points which $g$
moves arbitrarily far. Then extend g to the completion.

An instructive variant goes as follows: apply Lemma
\ref{L:ind}(\ref{ind_ex}) to $\Q\U$ to construct an automorphism $g$
which moves maximally. Then observe that $g$ is unbounded. Indeed, let
$a$ and $x$ be two elements of distance $N$. Choose a realisation $x'$
of $\tp(x/a)$ which is moved maximally by $g$. We have then
$d(x',g(x'))=d(a,g(a))+2N$.
\bigskip

For the sake of readability we now write $x^g$ for the image of a
point $x$ under an automorphism $g$.

We need some lemmas in preparation for the proof of Proposition~\ref{P:max}.

\begin{lemma}[Minimal distance amalgamation]\label{L:min_amalg}
 Let $(X, d)$ be a finite metric space and $(X\cup\{y\}, d_1)$,
 $(X\cup\{z\}, d_2)$ two extensions. Then there is a metric $\hat d$
 on $X\cup\{y,z\}$ extending $d_1, d_2$ with \[\hat d(y,z) =
 \max_{x\in X}|d_1(y,x) - d_2(z,x)|,\] where we identify
 $y$ and $z$ if $\hat d(y,z)=0$.
\end{lemma}
\begin{proof}
  This is easy to check.
\end{proof}

We call a sequence $(x_0,x_1,\ldots,x_n)$ \emph{geodesic} if
\[d(x_0,x_n)=d(x_0,x_1)+d(x_1,x_2)+\ldots+d(x_{n-1},x_n).\] Note that
$(x_0,x_1,x_2)$ is geodesic if and only if $x_0$ is independent from
$x_2$ over $x_1$. $(x_0,x_1,x_2,x_3)$ is geodesic if and only $x_0$ is
independent from $x_3$ over $x_1;x_2$. This shows that the next two
lemmas are special cases of Proposition \ref{P:max}.

\begin{lemma}\label{L:reflection_prep}
  Let $g$ be an unbounded isometry of $\U$. Then for any points $x_1,
  x_2\in\U$ there are points $y\in\U$ with $(y,x_1,x_2)$ geodesic and
  $d(y, y^g)$ arbitrarily large.
\end{lemma}

\begin{proof}
  First, observe that there are points $z\in\U$ with both $d(z,x_1)$
  and $d(z,z^g)$ arbitrarily large. Indeed, once $d(z,z^g)$ is
  sufficiently large, one of $z$ or $z^g$ will do.  In particular, we
  may take $z\in\U$ with $d(z,x_1) > d(x_1,x_2)$ and $d(z, z^g )$
  arbitrarily large. Let $X =\{x_2,z\}$ and let $X\cup\{y\}$ be a
  metric extension of $X$ with $(y,x_2,z)$ isometric with
  $(x_1,z,x_2)$.  Applying Lemma \ref{L:min_amalg} to $X\cup\{y\}$ and
  $X\cup\{x_1\}$, we get a pseudometric d on $X\cup\{x_1,y\}$ with
  \begin{align*}
    d(x_1,y)& = \max\,\bigl(|d(x_1,x_2) - d(y,x_2)|, |d(x_1,z) -
    d(y,z)|\bigr)\\
    &= d(z,x_1) - d(x_1,x_2)
  \end{align*}
  Therefore we may take such a point $y\in\U$ and we see that
  $(y,x_1,x_2)$ is geodesic.  Furthermore $d(y^g,z^g) = d(y,z) =
  d(x_1,x_2)$ so as $d(z,z^g)$ goes to infinity, $d(y,y^g)$ goes to
  infinity as well.

\end{proof}

\begin{lemma}\label{L:reflection}
  Let $g$ be an unbounded isometry of $\U$, and $x\in\U$. Then there
  are points $z\in\U$ with $d(z,z^g)$ arbitrarily large, such that
  $(z,x,x^g,z^g)$ is geodesic.
\end{lemma}
\begin{proof}
  Applying Lemma \ref{L:reflection_prep} to $g\inv$ we find $y$ with
  $d(y,y^g)$ arbitrarly large and $(x,x^g,y^g)$ geodesic. The inequality
  \[d(y,y^g)\leq d(y,x)+d(x,x^g)+d(x^g,y^g)\]
  implies that $a=\frac{1}{2}(d(y,y^g)-d(x,x^g))$ is not larger than
  $b=d(y,x)$. We may assume that $a$ is not negative.

  Pick some point $z$ such that $(x,z,y)$ is geodesic
  with $d(z,x)=a$ and $d(z,y)=b-a$ and
  such that \[\indep{z}{\{x,y\}}{y^g}.\]
  The distance between $z$ and $y^g$ is the minimum of
  \[d(z,x)+d(x,y^g)=a+d(x,y^g)=a+d(x,y)+b\]
  and
  \[d(z,y)+d(y,y^g)=(b-a)+d(y,y^g).\]
  By the definition of $a$ these two values are equal,
implying that $(z,x,y^g)$ and hence $(z,x,x^g,z^g,y^g)$ are geodesic.
 Since
  \begin{align*}
    d(z,z^g)&\geq d(z,y^q)-d(y^g,z^g)\\
    &=((b-a)+d(y,y^g))-(b-a)\\
    &=d(y,y^g)
  \end{align*}
we see that $d(z,z^g)$ can become arbitrarily large.

\end{proof}

For $p=\tp(a/X)$ we let $d(p,X)=\min\{d(a,x)\colon x\in X\}$.

\begin{lemma}\label{L:type}
  Let $g$ be an unbounded isometry of\/ $\U$. Let $X$ be a nonempty
  finite set. Then there is some $e = e(X)\geq 0$ such that every type
  $p$ over $X$ has some realisation $y$ in $\U$ for which
  $d(y,y^g)\geq 2d(p,X) - e$.
\end{lemma}
\begin{proof}
  We will show that $e = 2\diam(X)$ suffices.  Let $p =\tp(a/X)$, and
  fix $x_0\in X$. Apply Lemma \ref{L:reflection} to find a geodesic of
  the form $(z,x_0,x_0^g,z^g)$ with \[d(z,x_0) > (\max_{x\in X} d(a,
  X)) +\diam(X)\] Then $d(z,x) > d(a,x)$ for all $x\in X$. Therefore,
  if we apply Lemma \ref{L:min_amalg} to the metric spaces
  $X\cup\{z\}$ and $X\cup\{a\}$, we get a realisation $y$ of $p$ in
  $\U$ such that
    \begin{align*}
      d(y,z)& = \max_{x\in X}\,(d(z, x) - d(y, x))\\
      &= d(z,x_1) - d(y,x_1)
    \end{align*}
    for some $x_1\in X$. We claim that \[d(y,y^g)\geq 2d(p,X) - 2
    \diam(X).\] Considering first the path $(z,y,y^g,z^g)$ and then
    the path $(z,x_0,x_0^g,z^g)$, we find
    \begin{align*}
      d(y,y^g)&\geq d(z,z^g) - 2d(y,z)\\ &= 2d(z,x_0) + d(x_0,x_0^g) -
      2d(y,z)\\ &\geq 2[d(z,x_0) - d(y,z)]
    \end{align*}

    Then
    considering the triangle $(z,x_0,x_1)$ we have
    \begin{align*}
      d(z,x_0) - d(y,z)&\geq [d(z,x_1) - d(x_0,x_1)] - [d(z,x_1) -
        d(y,x_1)]\\
      &= d(y,x_1) - d(x_0,x_1)\\
      &\geq d(p,X) - \diam(X)
    \end{align*}
    and thus $d(y,y^g)\geq 2[d(p,X) - \diam(X)]$, as claimed.
\end{proof}

\begin{definition}
  For a type $p=\tp(a/X)$ and $d\in\R_{\geq 0}$ we call the type
  $\tp(y/X)=\{d(y,x)=d(a,x)+d\colon x\in X\}$ the $d$-prolongation of
  $p$, and denote it by $p+d$ or $\tp(a/X)+d$.
\end{definition}

\begin{proof}[Proof of Proposition~\ref{P:max}]
  Let $g$ be an unbounded isometry of $\U$ and let $p$ be a
  type over the nonempty finite set $X$. Let $a$ be a realisation of $p$ with
  \[\indep{a}{X}{X^gX^{g\inv}}\]
  and let $q =\tp(a/XX^gX^{g\in})$. We will show that we can find a
  realisation $z$ of $q$ for which
  \begin{equation}\label{eq:proofprop_1}
    \indep{z}{X^g}{z^g}
  \end{equation}
  and we claim that $g$ moves this realisation maximally.

  We address the second point first. By Lemma
  \ref{L:ind}~(\ref{ind_krit}) it suffices to check that
  \[\indep{z}{X}{X^g} \;\; \text{ and } \;\;\indep{zX}{X^g}{z^g}.\]
  As $z$ is a realisation of $q$ we find $\indep{z}{X}{X^g}$ and
  $\indep{X}{X^g}{z^g}$. Therefore the condition \eqref{eq:proofprop_1}
  suffices.

  Now we take up the construction of the point $z$. Applying Lemma
  \ref{L:type} to the set $X'=XX^gX^{g\inv}$ we find $e\geq 0$ such
  that for any $d$ the prolongation $q + d$ has a realisation $y$
  satisfying \[d(y,y^g)\geq 2[d(p,X')+d] - e.\] In particular, if $d >
  e$ we have \[d(y,y^g) > d.\] Fix $d > e$ and a corresponding
  realisation $y$ of $q + d$. An application of Lemma
  \ref{L:min_amalg} yields a realisation $z$ of $p$ at distance $d$
  from $y$.  We may suppose also that \[\indep{z}{Xy}{y^g}.\] We claim
  that this point $z$ has the required properties, if $d$ is taken
  sufficiently large.

  First we show
  \begin{equation}\label{eq:proofprop_2}
    \indep{z}{X}{X^gX^{g\inv}}
  \end{equation}
  As $\indep{y}{X}{X^g}{X^{g\inv}}$ , for any $x'\in X^gX^{g\inv}$ we
  have some $x\in X$ so that $(y,x,x')$ is geodesic. But $(y,z,x )$ is
  also geodesic, so $(z,x,x')$ is geodesic. Claim
  \eqref{eq:proofprop_2} follows.

  Now we check
  \[\indep{z}{X^g}{z^g}.\]
  We first examine $d(z,y^g)$. By the choice of $z$, this is the
  minimum of the values $d(z,u) + d(u,y^g)$ where $u$ ranges over$
  X\cup\{y\}$. For $x\in X$ we have
  \begin{align*}
    d(z,x) + d(x,y^g)& = d(z,x) + d(x^{g\inv},y)\\
    &= d(z,x) + d(x^{g\inv},z) + d\\
    &\leq d + 2\max_{x'\in X'} d(x',z)
  \end{align*}
  Compare this with \[d(z,y) + d(y,y^g) = d + d(y,y^g).\] We may take
  $d(y,y^g) > 2\max_{x'\in X'} d(x',z)$, and then there will be some
  $x\in X$ for which
  \[(z,x,y^g) \text{ is geodesic.}\]

  As $\indep{x}{X^g}{y^g}$ we have $x'\in X^g$ such such that
  $(x,x',y^g)$ is geodesic, and then $(z,x',y^g)$ is geodesic, and our
  claim follows.
\end{proof}
\hspace{1em}

\begin{proof}[Proof of \ref{T:urysohn}]
  Given an unbounded isometry $g$ of $\U$ and an arbitrary isometry
  $f$ of $\U$. By Proposition \ref{P:max} $g$ moves maximally.
  Consider \[{\cal U}=(\U, f,g, d,\R, +,<)\] as a $2$-sorted
  structure, one sort given by the elements of $\U$ with isometries
  $f$ and $g$ and the other sort given by the reals, considered as a
  ordered abelian group with the distance function $d$. Fix a
  countable dense subset $D$ of $\U$.  By the L\"owenheim-Skolem
  Theorem (see \cite{TZ}, Theorem 2.3.1), $D$ can be extended to a
  countable elementary substructure ${\cal U'}=(\U', f',g', d',R,
  +,<)$ where $R$ is a countable ordered abelian group.  As an
  elementary substructure of ${\cal U}$, the $R$-metric space $\U'$
  will be isometric to $R\U$. Also $g'$ moves maximally.

  In view of Corollary~\ref{C:Urydense} we now may apply Theorem~\ref{T:main} to $R\U$ to conclude that
  \sloppy there are $h'_1,\ldots,h'_8\in\aut(R\U)$ such that
  $f\restriction_{R\U}=g'^{\,h'_1}\cdot\ldots g'^{\,h'_8}$. Since
  $R\U$ is dense in $\U$, there are extensions $h_1,\ldots,h_8$ of the
  $h'_1,\ldots,h'_8$ to isometries of $\U$ and we have
  $f=g^{h_1}\cdot\ldots g^{h_8}$ on $R\U$ and by density on all of
  $\U$.
\end{proof}

\section{Application to free amalgamation}\label{Sec:free}

In order to apply our main theorem to free amalgamation classes,
we first prove a lemma in a more general context:

\begin{lemma}\label{l:free}
  Suppose that $\M$ is a countable structure with a stationary
  independence relation. Assume the following additional hypothesis
  for finite subsets $X,A,B$ of $\M$: If $A$ and $B$ are independent
  over $X$ and $X'$ is a subset of $X$ with $(A\cup B)\cap X\subset
  X'$, then $A$ and $B$ are independent over $X'$.

  Suppose that $g\in\aut(\M)$ and there is no type $p$ over a finite set whose set of realisations in $M$ is infinite, and is fixed pointwise by $g$. Then
  there is some $h\in\aut(\M)$ such that the
  commutator $[g,h]$ moves maximally.
\end{lemma}

\begin{proof}
  Let us first note two general facts which do not depend on the
  additional hypothesis.
  \begin{enumerate}
  \item\label{acl=dcl} Any $1$-type over $X$ has either exactly one
    realisation or infinitely many.

    Proof: Let $p$ be a type over $X$ and $A$ a finite non-empty set
    of realisations. Consider a realisation $a$ of $p$ which is
    independent from $A$ over $X$. Then all elements of $A$ have the
    same type over $Xa$, which implies that either $a$ does not belong
    to $A$ or $A=\{a\}$.
  \item\label{nondcl} If $\tp(a/X)$ has infinitely many realisations
    and $a$ is independent from $X'$ over $X$, then also $\tp(a/X')$
    has infinitely many realisations.

    Proof: Let $a_1$ and $a_2$ by two different realisations of
    $\tp(a/X)$. Choose a realisation $a'_1a'_2$ of $\tp(a_1a_2/X)$
    which is independent from $X'$ over $X$. Then $a'_1$ and $a'_2$
    are two different realisations of $\tp(a/X')$.
  \end{enumerate}

  We build $h$ by a `back-and-forth' construction as the union of a
  chain of finite partial automorphisms. It is enough to show the
  following: Let $h'$ be already defined on the finite set $U$ and let
  $p$ be a type over the finite set $X$. Then $h'$ has an extension
  $h$ such that $[g,h]$ moves $p$ maximally.

  If $p$ has only one realisation $a$, then $a$ is
  independent  over $X$ from every extension of $X$. So every automorphism moves
  $p$ maximally. So by (\ref{acl=dcl}) we may assume that $p$ has
  infinitely many realisations. By extending $h'$ if necessary we may
  also assume that $[g,h']$ is defined on $X$ and that
  \begin{equation}\label{eq:UV}
  {h'}\inv gh'(X)\subset U.
  \end{equation}

  Choose a realisation $a$ of $p$ which is independent from
  $X'=U\cup g(X)\cup [g,h'](X)$ over $X$. By (\ref{nondcl}) $\tp(a/X')$ has
  infinitely many realisations. So by the
  assumption on $g$ we can find such a realisation $a$ with
  $g(a)\not=a$.  Put $V=h'(U)$ and let $b$ realise $h'(\tp(a/U))$ in
  such a way that $\indep{b}{V}{g\inv(V)}$. Since $\tp(b/V)$ has
  infinitely many realisations, we can again assume that
  $g(b)\not=b$. Extend $h'$ to $Ua$ by setting $h'(a)=b$. Finally
  realise ${h'}\inv(\tp(g(b)/Vb))$ by $c$ such that
  \begin{equation}\label{eq:c1}
    \indep{c}{Ua}{g(a)g(X)}
  \end{equation}
  and extend $h'$ by setting $h'(c)=g(b)$.  We then have
  $[g,h'](a)=g\inv {h'}\inv g(b)=g\inv(c)$.\\

  \noindent Claim: $a$ is independent from $[g,h'](a)$ over
  $X;[g,h'](X)$.\\

  \noindent Proof: We know that $\indep{a}{X}{[g,h'](X)}$. So by Lemma
  \ref{L:ind}(\ref{ind_krit}) it remains to show that
  \begin{equation}\label{eq:claim}
    \indep{[g,h'](a)}{[g,h'](X)}{aX}.
  \end{equation}
  Since $g(b)\not=b$, $c$ is different from $a$. Also $a$ does not
  occur in $g(a)g(X)$. So by the additional hypothesis and
  \eqref{eq:c1} we have
  \begin{equation}\label{eq:c2}
    \indep{c}{U}{g(a)g(X)}.
  \end{equation}
  $\indep{a}{X}{U}$ implies $\indep{b}{h'(X)}{V}$. This together with
  $\indep{b}{V}{g\inv(V)}$ and $h'(X)\subset V$
  gives \[\indep{b}{h'(X)}{g\inv(V)}.\] Since independence is
  invariant under automorphisms, application of ${h'}\inv g$
  yields \[\indep{c}{{h'}\inv gh'(X)}{U}.\] From this and
  \eqref{eq:UV}, \eqref{eq:c2} we conclude \[\indep{c}{{h'}\inv
    gh'(X)}{g(a)g(X)}.\] An application of $g\inv$ now yields
  \eqref{eq:claim}.
\end{proof}

Using a result from \cite{MT}, we obtain the following, sharpening the main
theorem there.
\begin{corollary}\label{c:free}
  If $\M$ is the \Frl of a free amalgamation class in a relational
  language and not an indiscernible set, then if $\aut(\M)$ is
  transitive it is simple. For any nontrivial $g\in\aut(\M)$, every
  element can be written as a product of at most $16$ conjugates of
  $g$ and $g\inv$.
\end{corollary}
\begin{proof}
  It is easy to see that $\M$ satisfies the additional hypothesis of
  Lem\-ma~\ref{l:free}. Furthermore it was proved in \cite{MT},
  Corollary 2.10, that any element of $\aut(\M)$ satisfies the
  assumption on $g$. Hence the corollary follows directly from
  Lemma~\ref{l:free} via Corollary~\ref{C:main} and Example 2.2 (c).
\end{proof}

A small change in the proof of Lemma \ref{l:free} shows the following.

\begin{lemma}\label{L:am}
  Suppose that $\M$ is a countable structure with a stationary
  independence relation. Assume that $g$ moves  \emph{almost maximally}
  i.e.\ every $1$-type over a finite set $B$ has a realisation $b$
  which is independent from $g(b)$ over $B$. Then there is some $h\in\aut(\M)$
  such that the commutator $[g,h]$ moves maximally.
\end{lemma}
\begin{proof}
  We can follow the proof of \ref{l:free}, but we need not concern ourselves with whether $g$ fixes  $a$ or $b$. Instead we note that we can assume that
  $g(b)$ is independent from $b$ over $Vg\inv(V)$. It follows that
  $\indep{g(b)}{V}{b}$, which implies $\indep{c}{U}{a}$. Now
  \eqref{eq:c1} implies \eqref{eq:c2} by transitivity.
\end{proof}

This now implies:

\begin{corollary}\label{C:16}
  Suppose that $\M$ is a countable structure with a stationary
  independence relation and let $g\in\G$ move \emph{almost maximally}.
  Then any element of $\G$
  is the product of sixteen conjugates of $g$.
\end{corollary}

\noindent
{\bf Acknowledgement:} We thank Greg Cherlin and the referee for thoughtful
comments.



\vspace{5cm}

\begin{small}
\noindent\parbox[t]{15em}{
Katrin Tent,\\
Mathematisches Institut,\\
Universit\"at M\"unster,\\
Einsteinstrasse 62,\\
D-48149 M\"unster,\\
Germany,\\
{\tt tent@math.uni-muenster.de}}
\hfill\parbox[t]{18em}{
Martin Ziegler,\\
Mathematisches Institut,\\
Albert-Ludwigs-Universit\"at Freiburg,\\
Eckerstr. 1,\\
D-79104 Freiburg,\\
Germany,\\
{\tt ziegler@uni-freiburg.de}}

\end{small}
\end{document}